\documentclass[11pt,a4paper]{amsart}

\usepackage{amssymb}
\usepackage{amsmath}
\usepackage{diagrams}
\usepackage{hyperref}
\hypersetup{%
  colorlinks,
  bookmarksopen,
  bookmarksnumbered,
  citecolor=blue,
  linkcolor=black,
  pdfstartview=FitH,
}

\newtheorem{theorem}{Theorem}[section]
\newtheorem{lemma}[theorem]{Lemma}
\newtheorem{proposition}[theorem]{Proposition}
\newtheorem{corollary}[theorem]{Corollary}
\theoremstyle{definition}
\newtheorem{definition}[theorem]{Definition}
\newtheorem{remark}[theorem]{Remark}
\newtheorem{example}[theorem]{Example}
\newtheorem{question}[theorem]{Question}

\DeclareMathOperator{\Reg}{Reg}
\DeclareMathOperator{\VN}{VN}
\DeclareMathOperator{\id}{id}
\DeclareMathOperator{\spn}{span}
\newcommand{\schur}{\mathbin{\star}}
\newcommand{\pw}{\mathbin{\star}}
\newcommand\Gr{\mathop{\rm Gr}}
\newcommand\cb{\mathop{\rm cb}}
\newcommand\eh{\mathop{\rm eh}}
\let\nph\varphi
\let\varphi\phi
\let\phi\nph

\def\dc#1{\expandafter\def\csname#1\endcsname{\mathcal{#1}}}
\def\db#1{\expandafter\def\csname b#1\endcsname{\mathbb{#1}}}
\def\df#1{\expandafter\def\csname f#1\endcsname{\mathfrak{#1}}}

\def\loopy#1#2{%
  \def#1##1{\def\next{#2{##1}#1}\ifx##1\relax\let\next\relax\fi\next}}

\loopy{\makemathcals}{\dc}
\loopy{\makemathbbs}{\db}
\loopy{\makemathfraks}{\df}
\let\sect\S
\makemathbbs DBCRNQZET\relax
\makemathcals QWERTYUIOPASDFGHJKLZXCVBNM\relax
\makemathfraks QWERTYUIOPASDFGHJKLZXCVBNM\relax
\newcommand{\qtext}[1]{\quad\text{#1}\quad}
\newcommand{\qand}{\qtext{and}}

\begin{document}

\title{Schur multipliers of Cartan pairs}

\author{R.~H.~Levene, N.~Spronk, I.~G.~Todorov and L.~Turowska}

\address{School of Mathematical Sciences, University College Dublin, Belfield, Dublin 4, Ireland}
\email{rupert.levene@ucd.ie}

\address{Department of Pure Mathematics, University of Waterloo, Waterloo, Ontario N2L 3G1, Canada}
\email{nspronk@uwaterloo.ca}

\address{Pure Mathematics Research Centre, Queen's University Belfast, Belfast BT7 1NN, United Kingdom}
\email{i.todorov@qub.ac.uk}

\address{
  Department of Mathematics, Chalmers University of Technology and the University of Gothenburg,
Sweden}
\email{turowska@chalmers.se}

\date{9 September 2014
  \\\indent 2010~\emph{Mathematics subject classification:} Primary 46L10, Secondary 46L07.
  \\\indent \emph{Keywords:} Schur multiplier, Cartan masa, bimodule map} %

\begin{abstract}
  We define the Schur multipliers of a separable von Neumann
  algebra~$\M$ with Cartan masa~$\A$, generalising the classical Schur
  multipliers of~$\B(\ell^2)$.  We characterise these as the normal
  $\A$-bimodule maps on $\M$. If $\M$ contains a direct summand
  isomorphic to the hyperfinite II$_1$ factor, then we show that the
  Schur multipliers arising from the extended Haagerup tensor product
  $\A\otimes_{\eh}\A$ are strictly contained in the algebra of all
  Schur multipliers.
\end{abstract}

\maketitle
\setcounter{tocdepth}{1}
\tableofcontents

\section{Introduction}\label{s_intro}

Let~$\B(\ell^2)$ denote the space of bounded linear operators
on~$\ell^2$. The Schur multipliers of $\B(\ell^2)$ have attracted
considerable attention in the literature. These are the (necessarily
bounded) maps of the form \[M(\phi)\colon \B(\ell^2)\to
\B(\ell^2),\quad T\mapsto \phi\ast T\] where $\phi =
(\phi(i,j))_{i,j\in \bN}$ is a fixed matrix with the property that the
Schur, or entry-wise, product $\phi\ast T$ is in~$\B(\ell^2)$ for
every~$T\in \B(\ell^2)$. Here we identify operators in $\B(\ell^2)$
with matrices indexed by $\bN\times\bN$ in a canonical way.  It is
well-known that if~$\phi$ is itself the matrix of an element
of~$\B(\ell^2)$, then $M(\phi)$ is a Schur multiplier, but that not
every Schur multiplier of~$\B(\ell^2)$ arises in this way.

In fact~\cite{pa}, Schur multipliers are precisely the normal
(weak*-weak* continuous) $\D$-bimodule maps on $\B(\ell^2)$,
where~$\D$ is the maximal abelian selfadjoint algebra, or masa,
consisting of the operators in~$\B(\ell^2)$ whose matrix is
diagonal. By a result of R. R. Smith~\cite{smith}, each of these maps
has completely bounded norm equal to its norm as linear
map on~$\B(\ell^2)$. Moreover, it follows from a classical
result of A.~Grothendieck~\cite{Gro} that the space of Schur multipliers
of~$\B(\ell^2)$ can be identified with~$\D\otimes_{\eh}\D$,
where~$\otimes_{\eh}$ is the weak* (or extended) Haagerup tensor
product introduced by D.~P.~Blecher and R.~R.~Smith
in~\cite{blecher_smith}.

Recall \cite[Definition~3.1]{fm2} that a masa~$\A$ in a von Neumann
algebra~$\M$ is a Cartan masa if there is a faithful normal
conditional expectation of~$\M$ onto~$\A$, and the set of unitary
normalizers of~$\A$ in~$\M$ generates~$\M$.

Let $\R$ be the hyperfinite {\rm II}$_1$-factor. For each Cartan
masa~$\A\subseteq \R$, F.~Pop and R.~R.~Smith defined a Schur product
$\schur_\A\colon \R\times \R\to \R$ using the Schur products of finite
matrices and approximation techniques~\cite{ps}. Using this product,
they showed that every bounded $\A$-bimodule map $\R\to\R$ is
completely bounded, with completely bounded norm equal to its norm.
The separable von Neumann algebras~$\M$ containing a Cartan masa~$\A$
were coordinatised by J.~Feldman and C.~C.~Moore~\cite{fm1,fm2}. We use
this coordinatisation to define the Schur multipliers of~$(\M,\A)$. 
Our definition
generalises the classical notion of a Schur multiplier
of~$\B(\ell^2)$, and for~$\M=\R$ and certain
masas~$\A\subseteq \R$, our definition of Schur
multiplication extends the Schur product~$\schur_\A$ of~\cite{ps}.

In fact, the Schur multipliers of~$\M$ turn out to be the adjoints of
the multipliers of the Fourier algebra of the groupoid underlying the
von Neumann algebra~$\M$ (see~\cite{rbook,r}). Our focus, however, is
on algebraic properties such as idempotence, characterisation problems
and connections with operator space tensor products, so we restrict
our attention to Schur multipliers of von Neumann algebras with Cartan
masas.

Our main results are as follows. Let~$\M$ be a separable von Neumann
algebra with a Cartan masa~$\A$.  After defining the Schur multipliers
of~$(\M,\A)$, we show in Theorem~\ref{th_main} that these are
precisely the normal $\A$-bimodule maps $\M\to \M$, generalising the
well-known result for $\M=\B(\ell^2)$, $\A=\D$.  However,
if~$\M\ne\B(\ell^2)$, then the extended Haagerup tensor product
$\A\otimes_{\eh}\A$ need not exhaust the Schur multipliers; indeed we
show in that if~$\M$ contains a direct summand isomorphic to~$\R$,
then $\A\otimes_{\eh}\A$ does not contain every Schur multiplier
of~$\M$. This is perhaps surprising, since in~\cite{ps} Pop and Smith
show that every (completely) bounded $\A$-bimodule map on~$\R$ is the
weak* pointwise limit of transformations corresponding to elements of
$\A\otimes_{\eh}\A$.  Our result is a corollary to
Theorem~\ref{th_ch}, in which we show that there are no non-trivial
idempotent Schur multipliers of Toeplitz type on~$\R$ that come from
$\A\otimes_{\eh}\A$.

\subsection*{Acknowledgements}
The authors are grateful to Adam Fuller and David Pitts for providing
Remark~\ref{r_autcb} and drawing our attention to~\cite{cpz}. We also
wish to thank Jean Renault for illuminating discussions during the
preparation of this paper.

\section{Feldman-Moore relations and Cartan pairs}\label{s_prel}

Here we recall some preliminary notions and results from the work of
Feldman and Moore~\cite{fm1,fm2}. Throughout, let~$X$ be a set and
let~$R\subseteq X\times X$ be an equivalence relation on~$X$. We write
$x\sim y$ to mean that $(x,y)\in R$. 
For $n\in \bN$ with $n\ge2$, we
write
\[R^{(n)}=\{(x_0,x_1,\dots,x_n)\in X^{n+1}\colon x_0\sim
x_1\sim\dots\sim x_n\}.\] The $i$th coordinate projection of~$R$
onto~$X$ will be written as $\pi_i\colon R\to X$, $(x_1,x_2)\mapsto
x_i$.

\begin{definition}
  A map~$\sigma\colon
  R^{(2)}\to \bT$ is a \emph{$2$-cocycle on~$R$} if
  \[
  \sigma(x,y,z)\sigma(x,z,w)=\sigma(x,y,w)\sigma(y,z,w)
  \] for %
  all $(x,y,z,w)\in R^{(3)}$.  We say~$\sigma$
  is \emph{normalised} if~$\sigma(x,y,z)=1$ whenever two of $x$, $y$
  and~$z$ are equal. By \cite[Proposition~7.8]{fm1}, any normalised
  $2$-cocycle $\sigma$ is \emph{skew-symmetric}: for every
  permutation~$\pi$ on three elements,
  \[\sigma(\pi(x,y,z))=
  \begin{cases}
    \sigma(x,y,z)&\text{if $\pi$ is even},\\
    \sigma(x,y,z)^{-1}&\text{if $\pi$ is odd}.
  \end{cases}\]
\end{definition}

\begin{definition}
  An equivalence relation~$R$ on~$X$ is \emph{countable} if for every
  $x\in X$, the equivalence class $[x]_R=\{y\in X\colon x\sim y\}$ is
  countable.
\end{definition}

Now let~$(X,\mu)$ be a standard Borel probability space and suppose
that~$R$ is a countable equivalence relation which is also a Borel
subset of~$X\times X$, when~$X\times X$ is equipped with the product
Borel structure.

\begin{definition}
  For $\alpha\subseteq X$, let~$[\alpha]_R=\bigcup_{x\in \alpha}[x]_R$
  be the $R$-saturation of~$\alpha$. We say that~$\mu$ is
  \emph{quasi-invariant under~$R$} if
  \[ \mu(\alpha)=0\iff \mu([\alpha]_R)=0\] for any measurable
  set~$\alpha\subseteq X$.
\end{definition}

\begin{definition}
  We say that~$(X,\mu,R,\sigma)$ is a \emph{Feldman-Moore relation} if
  $(X,\mu)$ is a standard Borel probability space, $R$ is a countable
  Borel equivalence relation on~$X$ so that~$\mu$ is quasi-invariant
  under~$R$, and~$\sigma$ is a normalised $2$-cocycle on~$R$. When the
  context makes this unambiguous, for brevity we will simply refer to
  this Feldman-Moore relation as~$R$.
\end{definition}

Fix a Feldman-Moore relation~$(X,\mu,R,\sigma)$.

\begin{definition}
  Let~$E\subseteq R$ and let~$x,y\in X$. The horizontal slice of~$E$
  at~$y$ is
  \[ %
  E_y=\{z\in X\colon (z,y)\in E\}\times \{y\}\]
  and the vertical slice of~$E$ at~$x$ is
  \[E^x=\{x\}\times \{z\in X\colon (x,z)\in E\}.\] %
  We define
  \[ \bB(E)=\sup_{x,y\in X} |E_x|+|E^y|,\] and say that~$E$ is
  \emph{band limited} if~$\bB(E)<\infty$. We call a bounded Borel
  function~$a\colon R\to \bC$ \emph{left finite} if the support of~$a$
  is band limited, and we write
  \[ \Sigma_0=\Sigma_0(R)\] for the set of all such left finite
  functions on~$R$.
\end{definition}

\begin{definition}
  Equip~$R$ with the relative Borel structure from~$X\times X$. The
  \emph{right counting measure} for~$R$ is the measure~$\nu$ on~$R$
  defined by
  \[ \nu(E)=\int_X |E_y|\,d\mu(y)\] for each measurable
  set~$E\subseteq R$.  
\end{definition}

We shall also need a generalisation of the counting measure
$\nu$. For~$n\ge2$, let $\pi_{n+1}$ be the projection of $R^{(n)}$
onto~$X$ defined by 
$\pi_{n+1}(x_0,x_1,\ldots, x_n)=x_{n}$, and let~$\nu^{(n)}$ be the
measure on~$R^{(n)}$ given by
\[\nu^{(n)}(E)=\int_X|\pi_{n+1}^{-1}(y)\cap E|\,d\mu(y).\]

Now consider the Hilbert space~$H=L^2(R,\nu)$, where~$\nu$ is the right
counting measure of~$R$.

\begin{definition}
  We define a linear map
  \[ L_0\colon \Sigma_0\to \B(H),\qquad L_0(a)\xi:=a*_\sigma \xi\]
  for $a\in \Sigma_0$ and $\xi\in H$, where
  \begin{equation} a *_\sigma \xi(x,z)=\sum_{y\sim x}
    a(x,y)\xi(y,z)\sigma(x,y,z),\quad\text{for~$(x,z)\in
      R$} \label{eq:starsigma}.\end{equation} As shown in~\cite{fm2},
  this defines a bounded linear operator $L_0(a)\in \B(H)$ with
  $\|L_0(a)\|\leq \bB(E)\|a\|_\infty$, where $E$ is the support
  of~$a$.
\end{definition}

\begin{definition}
  We define
  \[ \M_0(R,\sigma)= L_0(\Sigma_0)\] to be the range
  of~$L_0$. %
\end{definition}

\begin{definition}
  The von Neumann algebra~$\M(R,\sigma)$ of the Feldman-Moore
  relation $(X,\mu,R,\sigma)$ is
  the von Neumann subalgebra of~$\B(H)$ generated
  by~$\M_0(R,\sigma)$. We will abbreviate this as~$\M(R)$ or
  simply~$\M$ where the context allows.
\end{definition}

  Let~$\Delta=\{(x,x)\colon x\in X\}$ be the diagonal of~$R$, and let
  $\chi_\Delta\colon R\to \bC$ be the characteristic function
  of~$\Delta$. Note that $\chi_\Delta$ is a unit vector in~$H$, since
  $\nu(\Delta)=\mu(X)=1$.

\begin{definition}
  The \emph{symbol map} of~$R$ is the map
  \[ s\colon \M\to H,\quad T\mapsto T\chi_\Delta.\]
  The \emph{symbol set} for~$R$ is the range of~$s$:
  \[ \Sigma(R,\sigma)=s(\M).\] We often abbreviate this
  as~$\Sigma(R)$ or~$\Sigma$.
\end{definition}

Since~$\sigma$ is normalised, equation~\eqref{eq:starsigma} gives
\begin{equation}\label{eq_newn}
  s(L_0(a))=a\quad\text{for $a\in \Sigma_0$,} 
\end{equation}
where equality holds almost everywhere. So we may view the Borel
functions~$a\in \Sigma_0$ as elements of~$H=L^2(R,\nu)$.  Moreover,
for~$T\in \M$ we have $\|s(T)\|_{\infty}\leq \|T\|$
by~\cite[Proposition~2.6]{fm2}. Hence
  \begin{equation}\label{eq:sigma0-inclusion}
    \Sigma_0\subseteq \Sigma\subseteq H\cap L^\infty(R,\nu).
  \end{equation}
  \begin{definition}
  By~\cite{fm2}, $s$ is a bijection onto~$\Sigma$, and its inverse
  \[ L\colon \Sigma\to \M\] extends~$L_0$. We call~$L$ the
  \emph{inverse symbol map} of~$R$. In fact, for any $a\in \Sigma$ we
  have $L(a)\xi=a*_\sigma \xi$ where $*_\sigma$ is the
  convolution product  formally defined by equation~\eqref{eq:starsigma}.
\end{definition}

If we equip~$\Sigma$ with the involution $a^*(x,y)=\overline{a(y,x)}$,
the pointwise sum and the convolution product~$*_\sigma$, then $s$ is
a $*$-isomorphism onto~$\Sigma$: for all $a,b\in \Sigma$ and
$\lambda,\mu\in\bC$, we have
\begin{align*}s(L(a)^*)(x,y)&=\overline{a(y,x)},\\ s(L(\lambda
  a)+L(\mu b))&=\lambda a+\mu b\quad\text{and}\\ s(L(a)
  L(b))&=a*_\sigma b.
\end{align*}
This is proven in~\cite{fm2}. By equation~(\ref{eq_newn}),
$\Sigma_0(R)$ is a $*$-subalgebra of~$\Sigma$, so $\M_0(R,\sigma)$ is
a $*$-subalgebra of~$\M(R,\sigma)$.

\begin{definition}
  Given~$\alpha\in L^\infty(X,\mu)$, let~$d(\alpha)\colon R\to \bC$ be given by
  \[ d(\alpha)(x,y)=
  \begin{cases}
    \alpha(x)&\text{if~$x=y$,}\\
    0&\text{otherwise}.
  \end{cases}\] Clearly $d(\alpha)\in \Sigma_0$. We write
  $D(\alpha)=L(d(\alpha))\in\M$, and we define the \emph{Cartan masa
    of~$R$} to be
  \[ \A=\A(R)=\{ D(\alpha)\colon \alpha\in L^\infty(X,\mu)\}.\]
  By~\cite{fm2},~$\A(R)$ is a Cartan masa in the von Neumann
  algebra~$\M(R,\sigma)$.

  Note that if $\xi\in H$ and $(x,y)\in R$, then
  \begin{eqnarray*}
    D(\alpha)\xi (x,y) & = & \sum_{z\sim x} d(\alpha)(x,z)\xi(z,y) \sigma(x,z,y) =
    \alpha(x) \xi(x,y) \sigma(x,x,y)\\
    & = & \alpha(x) \xi(x,y).
  \end{eqnarray*}
  Since this does not depend on the normalised $2$-cocycle ~$\sigma$,
  this shows that~$\A(R)$ does not depend on~$\sigma$.
\end{definition}

\begin{definition}
  If~$\A$ is a Cartan masa in a von Neumann algebra~$\M$, then we say
  that~$(\M,\A)$ is a \emph{Cartan pair}. If~$\M\subseteq \B(H)$
  where~$H$ is a separable Hilbert space, then we say that~$(\M,\A)$
  is a \emph{separably acting} Cartan pair.

  We say that two Cartan pairs~$(\M_1,\A_1)$ and $(\M_2,\A_2)$ are
  isomorphic, and write~$(\M_1,\A_1)\cong(\M_2,\A_2)$, if there is a
  $*$-isomorphism of~$\M_1$ onto~$\M_2$ which carries $\A_1$
  onto~$\A_2$.

  A \emph{Feldman-Moore coordinatisation} of a Cartan pair~$(\M,\A)$
  is a Feldman-Moore relation $(X,\mu,R,\sigma)$ so that
  \[ (\M,\A)\cong (\M(R,\sigma),\A(R)).\]
\end{definition}

\begin{definition}\label{def:isorel}
  For $i=1,2$, let $R_i=(X_i,\mu_i,R_i,\sigma_i)$ be a Feldman-Moore
  relation with right counting measure $\nu_i$.  We say that these are
  isomorphic, and write $R_1\cong R_2$, if there is a Borel
  isomorphism $\rho\colon X_1\to X_2$ so that
  \begin{enumerate}
  \item $\rho_*\mu_1$ is equivalent to~$\mu_2$, where
    $\rho_*\mu_1(E)=\mu_1(\rho^{-1}(E))$ for $E\subseteq X_2$;
  \item $\rho^2(R_1)=R_2$, up to a $\nu_{2}$-null set, where $\rho^2=\rho\times\rho$; and
  \item $\sigma_2(\rho(x),\rho(y),\rho(z))=\sigma_1(x,y,z)$ for
    a.e.~$(x,y,z)\in R_1^{(2)}$ with respect to $\nu_1^{(2)}$. 
  \end{enumerate}
\end{definition}

Our definition of the Schur multipliers of a von Neumann algebra~$\M$
with a Cartan masa~$\A$ will rest on:

\begin{theorem}[The Feldman-Moore coordinatisation~{\cite[Theorem~1]{fm2}}]
  \label{thm:fmii1}
  Every separably acting Cartan pair~$(\M,\A)$ has a Feldman-Moore
  coordinatisation.
  Moreover, if $R_i=(X_i,\mu_i,R_i,\sigma_i)$ is a Feldman-Moore
  coordinatisation of~$(\M_i,\A_i)$ for $i=1,2$, then
  \[ (\M_1,\A_1)\cong (\M_2,\A_2)\iff R_1\cong R_2.\]
\end{theorem}

\begin{remark}\label{rk:unitary}
  Suppose that we have isomorphic Feldman-Moore relations $R_1$
  and~$R_2$, with an isomorphism~$\rho\colon X_1\to X_2$ as in
  Definition~\ref{def:isorel}.  A calculation shows that if~$h\colon
  X_2\to \bR$ is the Radon-Nikodym derivative of~$\rho_*\mu_1$ with
  respect to~$\mu_2$, then the operator \[U\colon
  L^2(R_2,\nu_2)\to L^2(R_1,\nu_1),\] given for $(x,y)\in
  R_1$ and $f\in L^2(R_2,\nu_2)$ by
  \[U(f)(x,y) = h(\rho(y))^{-1/2} f(\rho(x),\rho(y)),\] 
  is unitary. 
  Moreover, writing $L_i$ for the inverse
  symbol map of~$R_i$, for $a\in \Sigma_0(R_1,\sigma_1)$ we have
  \begin{equation}\label{eq:unitary-action}
    U^*L_1(a)U=L_2( a\circ \rho^{-2})
  \end{equation} 
  where
  \[\rho^{-2}(u,v)=(\rho^{-1}(u),\rho^{-1}(v)),\quad(u,v)\in R_2.\] It
  follows that
  \[U^*\M(R_1,\sigma_1)U=\M(R_2,\sigma_2)\quad\text{and}\quad
  U^*\A(R_1)U=\A(R_2),\] so conjugation by~$U$ implements an isomorphism
  \[ (\M(R_1,\sigma_1),\A(R_1))\cong (\M(R_2,\sigma_2),\A(R_2))\] whose
  existence is assured by Theorem~\ref{thm:fmii1}.
\end{remark}

\section{Algebraic preliminaries}\label{s_sap}

In this section, we collect some algebraic observations.
Fix a Feldman-Moore relation $R=(X,\mu,R,\sigma)$ with right counting
measure~$\nu$, let~$H=L^2(R,\nu)$, let~$\M=\M(R,\sigma)$ and
let~$\A=\A(R)$. Also let $\Sigma_0$ be the collection of left finite
functions on~$R$, and let~$s,L,\Sigma$ be the symbol
map, inverse symbol map and the symbol set of~$R$, respectively.

We can describe the bimodule action of~$\A$ on~$\M$ quite easily in
terms of the pointwise product of symbols.
\begin{definition}
  For~$a,b\in L^\infty(R,\nu)$, let $ a\pw b$ be the pointwise product
  of~$a$ and~$b$.
\end{definition}

\begin{definition}
  For~$\alpha\in L^\infty(X,\mu)$ we write
  \[c(\alpha)\colon R\to\bC,\quad (x,y)\mapsto
  \alpha(x)\quad\text{and}\quad r(\alpha)\colon R\to\bC,\quad
  (x,y)\mapsto \alpha(y).\]
\end{definition}
\begin{lemma}\label{lem:action}
  For~$a\in \Sigma$ and~$\beta,\gamma\in L^\infty(X,\mu)$, we have
  \[ D(\beta)L(a)D(\gamma)=L(c(\beta)\pw a\pw r(\gamma)).\]
\end{lemma}
\begin{proof}
The statement follows from the identity
$s\big(D(\beta)L(a)D(\gamma))=c(\beta)\pw  a\pw r(\gamma)$;
its verification is straightforward, but we include it for completeness:
\begin{align*}
  s\big(D(\beta)L(a)D(\gamma)\big)(x,y)
  & = \big(D(\beta)L(a)D(\gamma)\chi_{\Delta}\big)(x,y) \\
  & = \beta(x)\big(L(a)D(\gamma)\chi_{\Delta}\big)(x,y)\\
  & = \beta(x)\sum_{z\sim x} a(x,z) \big(D(\gamma)\chi_{\Delta}\big)(z,y)\sigma(x,z,y)\\
  & = \beta(x)\sum_{z\sim y} a(x,z) \gamma(z) \chi_{\Delta}(z,y)\sigma(x,z,y)\\
  & = \beta(x) a(x,y) \gamma(y) \sigma(x,y,y)\\
  & = \beta(x) a(x,y) \gamma(y) \\
&= \left(c(\beta)\pw a\pw r(\gamma)\right)(x,y).\qedhere
\end{align*}
\end{proof}

Recall the standard way to associate an inverse semigroup to $R$.
Suppose that~$f\colon \delta\to \rho$ is a Borel isomorphism
between two Borel subsets~$\delta,\rho\subseteq X$. Such a map will be
called a \emph{partial Borel isomorphism of~$X$}. If~$g\colon
\delta'\to \rho'$ is another partial Borel isomorphism of~$X$, then we
can (partially) compose them as follows:
\[ g\circ f\colon f^{-1}(\rho\cap \delta')\to g(\rho\cap
\delta'),\quad x\mapsto g(f(x)).\] Let us write $\Gr f=\{(x,f(x))\colon
\text{$x$ is in the domain of~$f$}\}$ for the graph of~$f$. Under
(partial) composition, the set~\[ \I(R)=\{f\colon \text{$f$ is a partial Borel
  isomorphism of~$X$ with $\Gr f\subseteq R$}\}\] is an inverse
semigroup, where the inverse of~$f\colon\delta\to\rho$ in~$\I(R)$ is
the inverse function~$f^{-1}\colon\rho\to \delta$.

If~$f\in\I(R)$, then~$\bB(\Gr f)\leq 2$, so~$\chi_{\Gr f}\in
\Sigma_0$. We define an operator~$V(f)\in \M$ by
\[ V(f)=L(\chi_{\Gr f}).\]
If~$\delta$ is a Borel subset of~$X$, we will write
$P(\delta)=V(\id_{\delta})$ where~$\id_{\delta}$ is the identity map
on the Borel set~$\delta\subseteq X$. Note that $P(\delta) =
D(\chi_{\delta})$.

\goodbreak

\begin{lemma}\label{l_pin}\leavevmode
  \begin{enumerate}
  \item If~$f\in \I(R)$, then $V(f)^*=V(f^{-1})$.
  \item If~$f\in \I(R)$ and $\delta,\rho$ are Borel subsets of~$X$,
    then \[P(\delta)V(f)P(\rho)=V(\id_\rho\circ f\circ \id_\delta).\]
  \item If~$\delta$ is a Borel subset of~$X$, then $P(\delta)$ is a
    projection in~$\A$, and every projection in~$\A$ is of this form.
  \item If~$\rho$ is a Borel subset of~$X$, then $V(f) P(\rho) =
    P(f^{-1}(\rho))V(f)$.
  \item \label{pisom} If $f : \delta\rightarrow\rho$ is in~$\I(R)$,
    then~$V(f)$ is a partial isometry %
    with initial projection $P(\rho)$ and final projection $P(\delta)$.
  \end{enumerate}
\end{lemma}
\begin{proof}
(1) It is straightforward that $\chi_{\Gr (f^{-1})}=(\chi_{\Gr f})^*$
    (where the ${}^*$ on the right hand side is the involution
    on~$\Sigma$ discussed in~\sect\ref{s_prel} above). Since~$L$ is a $*$-isomorphism,
    $V(f^{-1})=V(f)^*$.

(2) Note that
    \[ (\delta\times X)\cap \Gr f\cap (X\times \rho) = \Gr(\id_\rho\circ f\circ \id_\delta),\]
    so
    \[ c(\chi_\delta)\pw \chi_{\Gr f}\pw
    r(\chi_\rho)=\chi_{\Gr(\id_\rho\circ f\circ \id_\delta)}.\]
    By Lemma~\ref{lem:action},
    \begin{equation*}
      P(\delta)V(f)P(\rho)
      =L(c(\chi_\delta)\pw \chi_{\Gr f}\pw r(\chi_\rho))
      = V(\id_\rho\circ f\circ \id_\delta).
    \end{equation*}

(3) Taking $f=\id_\delta$ in~(1) shows
    that~$P(\delta)=V(\chi_{\id_\delta})$ is self-adjoint; and taking
    $f=\id_\Delta$ and $\delta=\rho$ in~(2) shows that $P(\delta)$ is
    idempotent.
    So $P(\delta)$ is a projection.  Since
    $P(\delta)=D(\chi_{\delta})$, we have $P(\delta)\in
    \A$. Conversely, since~$L$ is a $*$-isomorphism, any
    projection~$P$ in~$\A$ is equal to~$D(\alpha)$ for some
    projection~$\alpha\in L^\infty(X,\mu)$. So $\alpha=\chi_\delta$ for some Borel
    set~$\delta\subseteq X$, and hence $P = P(\delta)$ for some Borel set $\delta\subseteq X$.

    (4) Since $\id_\rho\circ f=f\circ\id_{f^{-1}(\rho)}$ and~$P(X)=I$,
    this follows by taking~$\delta=X$ in~(2).

    (5) Using the fact that $\sigma$ is normalised, a simple
    calculation yields
    \[ \chi_{\Gr f}*_\sigma \chi_{\Gr f^{-1}} = \chi_{\Gr(\id_\delta)}.\]
    Applying the $*$-isomorphism~$L$ and using~(1) gives
    $V(f)V(f)^*=P(\delta)$
    and replacing~$f$ with $f^{-1}$ gives
    $V(f)^*V(f)=P(\rho)$.
\end{proof}

\begin{proposition}\label{prop:bimod-symb}
  Let~$\Phi\colon \M\to \M$ be a linear $\A$-bimodule map.
  \begin{enumerate}
  \item If $f\in \I(R)$ and~$V=V(f)$, then $s(\Phi(V))=
    \chi_{\Gr f} \pw s(\Phi(V))$.
  \item For $i=1,2$, let $f_i\colon \delta_i\to \rho_i$ be in~$\I(R)$
    and let~$V_i=V(f_i)$. If $G=\Gr(f_1)\cap \Gr(f_2)$, then
    \[ \chi_G\pw s(\Phi(V_1))= \chi_G\pw s(\Phi(V_2)).\]
  \end{enumerate}
\end{proposition}
\begin{proof}
(1) Let~$f\in \I(R)$ and let $\rho\subseteq X$ be a Borel set.
    Since $\Phi$ is an $\A$-bimodule map, Lemma~\ref{l_pin} implies
    that
    \begin{align*}V^*\Phi(V)P(\rho) &= V^*\Phi(V P(\rho))
      = V^*\Phi(P(f^{-1}(\rho))V) \\&=
      V^*P(f^{-1}(\rho))\Phi(V)\\&=
      (P(f^{-1}(\rho))V)^*\Phi(V) \\&=
      (VP(\rho))^*\Phi(V) = P(\rho) V^*\Phi(V).
    \end{align*}
    Hence $V^*\Phi(V)$ commutes with all projections in $\A$, and
    since $\A$ is a masa, $V^*\Phi(V)\in \A$. If~$\delta$ is the
    domain of~$f$, then by Lemma~\ref{l_pin}(\ref{pisom}), $P(\delta)$ is
    the final projection of~$V$, and therefore
    \[\Phi(V)=\Phi(P(\delta)V)=P(\delta)\Phi(V)=VV^*\Phi(V)\in
    V\A.\] So $\Phi(V)=VD(\gamma)$ for some $\gamma\in
    L^\infty(X,\mu)$. By Lemma~\ref{lem:action},
    \[ s(\Phi(V))=s(VD(\gamma))=s(L(\chi_{\Gr
      f})D(\gamma))=\chi_{\Gr f}\pw d(\gamma),\] so $s(\Phi(V))= \chi_{\Gr f} \pw
    s(\Phi(V))$.

    (2) Let~$\delta=\pi_1(G)$ where~$\pi_1(x,y)=x$ for $(x,y)\in
    R$. It is easy to see that $\chi_G=c(\chi_\delta)\pw \chi_{\Gr
      f_i}$ for $i=1,2$. By part~(1), $s(\Phi(V_i))=\chi_{\Gr f_i}\pw
    s(\Phi(V_i))$. Hence by Lemmas~\ref{lem:action}
    and~\ref{l_pin},
    \begin{align*}
      \chi_G \pw s(\Phi(V_i))&=c(\chi_\delta)\pw \chi_{\Gr f_i}\pw
      s(\Phi(V_i))\\&=c(\chi_\delta)\pw
      s(\Phi(V_i)) = s(P(\delta)\Phi(V_i))\\&=
      s(\Phi(P(\delta)V_i)) = s(\Phi(V(f_i\circ \id_\delta))).
  \end{align*}
  The definition of~$\delta$ ensures that $f_1\circ
  \id_\delta=f_2\circ \id_\delta$, so $\chi_G \pw s(\Phi(V_1)) =
  \chi_G\pw s(\Phi(V_2))$.
\end{proof}

\section{Schur multipliers: definition and characterisation}\label{s_sm}

Let $(X,\mu,R,\sigma)$ be a Feldman-Moore coordinatisation of a
separably acting Cartan pair~$(\M,\A)$, and let $\Sigma_0,\Sigma$ 
be as in Section~\ref{s_prel}. 
In this section we define the class $\fS(R,\sigma)$ of Schur
multipliers of the von Neumann algebra~$\M$ with respect to the
Feldman-Moore relation~$R$. The main result in this section,
Theorem~\ref{th_main}, characterises these multipliers as normal
bimodule maps. From this it follows that~$\fS(R,\sigma)$ depends only
on the Cartan pair~$(\M,\A)$. We also show that isomorphic
Feldman-Moore relations yield isomorphic classes of Schur multipliers.

\begin{definition}
  \label{d_sh}
  Let~$R=(X,\mu,R,\sigma)$ be a Feldman-Moore coordinatisation of a
  Cartan pair~$(\M,\A)$.  We say that $\phi\in L^\infty(R,\nu)$ is a
  \emph{Schur multiplier of~$(\M,\A)$ with respect to~$R$}, or simply
  a \emph{Schur multiplier of~$\M$}, if
  \[ a\in \Sigma(R,\sigma)\implies \phi\pw a\in \Sigma(R,\sigma) \]
  where~$\pw$ is the pointwise product on~$L^\infty(R,\nu)$.
  We then write
  \[ m(\phi)\colon \Sigma(R,\sigma)\to \Sigma(R,\sigma),\quad a\mapsto
  \phi\pw a\] and
  \[ M(\phi)\colon \M\to \M,\quad T\mapsto L (\phi\pw s(T)).\]
\end{definition}

Set
\[\fS = \fS(R,\sigma) = \{\phi\in L^\infty(R,\nu)\colon \text{$\phi$
 is a Schur multiplier of $\M$}\}.\]
It is clear from Definition~\ref{d_sh} that $\fS(R,\sigma)$ is an algebra
with respect to pointwise addition and multiplication of functions.

\begin{example}\label{ex_bh}
  For a suitable choice of Feldman-Moore coordinatisation,
  $\fS(R,\sigma)$ is precisely the set of classical Schur multipliers
  of~$\B(\ell^2)$. Indeed, let $X = \bN$, equipped with the (atomic)
  probability measure $\mu$ given by $\mu(\{i\}) = p_i$, $i\in \bN$,
  and set $R = X\times X$.  If~$p_i>0$ for every~$i\in \bN$,
  then~$\mu$ is quasi-invariant under~$R$. Let $\sigma$ be the trivial
  $2$-cocycle $\sigma \equiv 1$.  The right counting measure for the
  Feldman-Moore relation~$(X,\mu,R,\sigma)$ is $\nu=\kappa\times \mu$
  where~$\kappa$ is counting measure on~$\bN$. Indeed,
  for~$E\subseteq R$, 
  \[\nu(E)=\sum_{y\in\bN} |E_y| \mu(\{y\})
    =\sum_{y\in\bN} \kappa\times \mu(E_y)= \kappa\times \mu(E).\]
  Hence $L^2(R,\nu)$ is canonically isometric to the Hilbert space
  tensor product $\ell^2\otimes\ell^2(\bN,\mu)$. Let~$T\in
  \M(R,\sigma)$. For an elementary tensor $\xi\otimes\eta\in
  L^2(R,\nu)$, we have
  \[T(\xi\otimes\eta) (i,j) = L_{s(T)}(\xi\otimes \eta)(i,j) =
  \sum_{k=1}^{\infty} s(T)(i,k)\xi(k)\eta(j) = (A_{s(T)}\xi\otimes
  \eta)(i,j)\] where $A_a\in \B(\ell^2)$ is the operator with
  matrix~$a\colon \bN\times\bN\to \bC$. It follows that the map
  $T\mapsto A_{s(T)}\otimes I$ is an isomorphism between $\M(R,\sigma)$ and
  $\B(\ell^2)\otimes I$, so
  \[\Sigma(R,\sigma)=\{a\colon\bN\times\bN\to \bC\mid \text{$a$ is the matrix of~$A$
    for some~$A\in \B(\ell^2)$}\}.\] In particular, a function $\phi :
    \bN\times\bN\rightarrow \bC$ is in~$\fS(R,\sigma)$ if and only
    if $\phi$ is a (classical) Schur multiplier of~$\B(\ell^2)$.
\end{example}

\begin{example}\label{ex_d}
  If~$(X,\mu,R,\sigma)$ is a Feldman-Moore relation and $\Delta$ is
  the diagonal of~$R$, then $\chi_{\Delta}\in \fS(R,\sigma)$ since for
  any~$a\in L^\infty(R,\nu)$, the function 
  \[\chi_\Delta\pw a=d(x\mapsto a(x,x))\]
  belongs to $\Sigma_0$ and hence to $\Sigma$. 
\end{example}

More generally:

\begin{proposition}\label{prop:sigma0}
  For any Feldman-Moore relation~$(X,\mu,R,\sigma)$, we have
  $\Sigma_0(R,\sigma)\subseteq \fS(R,\sigma)$.
\end{proposition}
\begin{proof}
  Let~$\phi\in \Sigma_0(R,\sigma)$ and let~$a\in \Sigma(R,\sigma)$. 
  Recall that~$a\in L^\infty(R,\nu)$, so we can
  choose a bounded Borel function~$\alpha\colon R\to\bC$ with
  $\alpha=a$ almost everywhere with respect to~$\nu$.  The function
  $\phi\pw \alpha$ is then bounded, and its support is a subset of the
  support of~$\phi$, which is band limited. Hence $\phi\pw \alpha\in
  \Sigma_0$, and $\phi\pw \alpha=\phi\pw a$ almost everywhere. By
  equation~(\ref{eq:sigma0-inclusion}), we have $\phi\pw a\in
  \Sigma(R,\sigma)$, so $\phi\in \fS(R,\sigma)$.
\end{proof}

We now embark on the proof of our main result.

\begin{lemma}\label{lem:cgt}
  Let~$\fX$ be a Banach space, let~$V$ be a complex normed vector
  space, and let~$\alpha,\beta$ and $h$ be linear maps so that the
  following diagram commutes:
  \begin{diagram}
    \fX & \rTo^{h} & V \\
    \dTo^{\alpha{}} &                        &  \dTo_{\beta{}}\\
    \fX & \rTo^h & V
  \end{diagram}
If~$h$ and $\beta$ are continuous and~$h$ is injective, then~$\alpha$
is continuous.
\end{lemma}
\begin{proof}
  If $x_n\in \fX$ with $x_n\to 0$ and $\alpha(x_n)\to y$ as~$n\to \infty$ for some $y\in
  \fX$, then
  \begin{align*} h(y)=h(\lim_{n\to \infty} \alpha(x_n))&=\lim_{n\to
      \infty}h(\alpha(x_n))
    \\
    &=\lim_{n\to \infty}\beta(h(x_n))=\beta(h(\lim_{n\to
      \infty}x_n))=\beta(h(0))=0.
  \end{align*}
  Since~$h$ is injective, $y=0$ and $\alpha$ is continuous by the
  closed graph theorem.
\end{proof}

If~$\phi$ is a Schur multiplier of~$\M$, then we have the following
commutative diagram of linear maps:
\begin{diagram}
  \M & \pile{\lTo^L\\ \rTo_s} & \Sigma(R,\sigma) \\
  \dTo^{M(\phi)} &                        &  \dTo_{m(\phi)}\\
  \M & \pile{\lTo^L\\ \rTo_s} & \Sigma(R,\sigma)
\end{diagram}
We now record some continuity properties of this diagram.
\begin{proposition}\label{prop:continuity}\leavevmode
  Let~$(X,\mu,R,\sigma)$ be a Feldman-Moore relation,
  let~$(\M,\A)=(\M(R,\sigma),\A(R))$, let~$\H=L^2(R,\nu)$ where~$\nu$
  is the right counting measure of~$R$, and write
  $\Sigma=\Sigma(R,\sigma)$.  Let~$\phi\in \fS(R,\sigma)$.
  \begin{enumerate}
  \item $m(\phi)$ is continuous
    as a map on $(\Sigma,\|\cdot\|_\infty)$.
  \item \label{s-contraction} $s$ is a contraction from
    $(\M,\|\cdot\|_{\B(\H)})$ to $(\Sigma,\|\cdot\|_\infty)$.
  \item \label{Mphi-cts} $M(\phi)$ is norm-continuous.
  \item $m(\phi)$ is continuous as a map on $(\Sigma,\|\cdot\|_2)$.
  \item $s$ is a contraction from $(\M,\|\cdot\|_{\B(\H)})$
    to $(\Sigma, \|\cdot\|_2)$.
  \item\label{lem-obvious} $s$ is continuous from~$(\M,{\rm SOT})$ to
    $(\Sigma,\|\cdot\|_2)$, where SOT is the strong operator topology
    on~$\M$.
  \end{enumerate}
\end{proposition}
\begin{proof}\leavevmode
(1) and~(4) follow from the fact that~$\phi$ is essentially bounded.

(2) See~\cite[Proposition~2.6]{fm2}.

(3) This follows from~(\ref{s-contraction}) and Lemma~\ref{lem:cgt}.

(5) follows from the fact that $\chi_\Delta$ is a unit vector in~$\H$.

(6) Let $\{T_\lambda\}$ be a net in~$\M$ which converges
    in the SOT to~$T\in\M$. Then
    $s(T_\lambda)=T_\lambda(\chi_\Delta)\to T(\chi_\Delta)=s(T)$ in
    $\|\cdot\|_2$.\qedhere
\end{proof}

If~$R$ is a Feldman-Moore relation with right counting measure~$\nu$,
let~$\nu^{-1}$ be the measure on~$R$ given by
\[\nu^{-1}(E)=\nu(\{(y,x)\colon (x,y)\in E\}).\] We will need the
following facts, which are established in~\cite{fm2}.
\begin{proposition}\leavevmode\label{prop:nuinverse}
  \begin{enumerate}
  \item $\nu$ and $\nu^{-1}$ are mutually absolutely continuous;
  \item if~$d=\frac{d\nu^{-1}}{d\nu}$, then the
    set~$d^{1/2}\Sigma_0=\{d^{1/2}a\colon a\in \Sigma_0\}$ of
    \emph{right finite} functions on~$R$ has the property that
    for~$b\in d^{1/2}\Sigma_0$, the formula
    \[ R_0(b)\xi=\xi *_\sigma b,\quad \xi\in H \] defines a bounded
    linear operator~$R_0(b)\in\B(H)$; and
  \item for~$b\in d^{1/2}\Sigma_0$, we have $R_0(b)\in\M'$ and
    $R_0(b)(\chi_\Delta)=b$.
  \end{enumerate}
\end{proposition}
We will now see that the SOT-convergence of a \emph{bounded} net
in~$\M$ is equivalent to the $\|\cdot\|_2$ convergence of its image
under~$s$.

\begin{proposition}\label{p_sconv}
  Let $\{T_\lambda\}\subseteq \M(R)$ be a norm bounded net.

  \begin{enumerate}
  \item $\{T_\lambda\}$ converges in the {\rm SOT} if and only if
    $\{s(T_\lambda)\}$ converges with respect to~$\|\cdot\|_2$.
  \item  For~$T\in\M$, we have
    \[ T_\lambda\to_{{\rm SOT}} T\iff s(T_\lambda)\to_{\|\cdot\|_2} s(T).\]
  \end{enumerate}
\end{proposition}
\begin{proof} (1) The ``only if'' is addressed by
  Proposition~\ref{prop:continuity}(\ref{lem-obvious}).

  Conversely, suppose that $s(T_\lambda)=T_\lambda(\chi_{\Delta})$
  converges with respect to~$\|\cdot\|_2$ on~$H$.  For a right
  finite function $b\in d^{1/2}\Sigma_0$, we have
  \[R_0(b)T_\lambda (\chi_{\Delta}) = T_\lambda R_0(b)(\chi_\Delta) =
  T_\lambda (b)\] which converges
  in~$H$. By~\cite[Proposition~2.3]{fm2}, the set of right finite
  functions is dense in $H$. Since~$\{T_\lambda\}$ is bounded, we
  conclude that $T_\lambda(\xi)$ converges for every $\xi\in H$. So we
  may define a linear operator~$T\colon H\to H$
  by~$T(\xi)=\lim_\lambda T_\lambda(\xi)$; then~$\|T(\xi)\|\leq
  \sup_\lambda \|T_\lambda\| \|\xi\|$, so~$T\in\B(H)$. By
  construction, $T_\lambda\to T$ strongly.

  (2) The direction ``$\implies$'' follows from
  Proposition~\ref{prop:continuity}(\ref{lem-obvious}). For the converse, apply~(1) to see that
  if~$s(T_\lambda)\to_{\|\cdot\|_2} s(T)$, then $T_\lambda \to_{{\rm SOT}} S$
  for some~$S\in \M$. Hence $s(T_\lambda)\to _{\|\cdot\|_2} s(S)$; therefore
  $s(S)=s(T)$ and so $S=T$.
\end{proof}

The following argument is taken from the proof
of~\cite[Corollary~2.4]{ps}.

\begin{lemma}\label{lem-popsmith}
  Let~$H$ be a separable Hilbert space and $\M\subseteq\B(H)$ be a von
  Neumann algebra.  Suppose that~$\Phi \colon\M\to\M$ is a bounded
  linear map which is strongly sequentially continuous on bounded
  sets, meaning that for every~$r>0$, whenever $X,X_1,X_2,X_3,\dots$
  are operators in~$\M$ with norm at most~$r$ with $X_n\to _{\rm
    SOT}X$ as $n\to \infty$, we have $\Phi(X_n)\to_{{\rm
      SOT}}\Phi(X)$. Then $\Phi$ is normal.
\end{lemma}
\begin{proof}
  For~$\xi,\eta\in H$, let~$\omega_{\xi,\eta}$ be the vector
  functional in~$\M_*$ given by $\omega_{\xi,\eta}(X)=\langle
  X\xi,\eta\rangle$, $X\in \M$, and let
  \[K=\ker\Phi^*(\omega_{\xi,\eta})\qand K_r=K\cap \{X\in \M\colon
  \|X\|\leq r\},\ \ \ \text{for $r>0$}.\] Let~$r>0$. Since~$H$ is
  separable, $\M_*$ is separable and so the strong operator topology
  is metrizable on the bounded set~$K_r$. From the sequential strong
  continuity of~$\Phi$ on $\{ X\in\M\colon \|X\|\leq r\}$, it follows
  that~$K_r$ is strongly closed. Since~$K_r$ is bounded and convex,
  each $K_r$ is ultraweakly closed. By the Krein-Smulian theorem, $K$
  is ultraweakly closed, so $\Phi^*(\omega_{\xi,\eta})$ is ultraweakly
  continuous; that is, it lies in $\M_*$. The linear span of
  $\{\omega_{\xi,\eta}\colon \xi,\eta\in H\}$ is (norm) dense in~$\M_*$, so
  this shows that $\Phi^*(\M_*)\subseteq\M_*$.
  Define $\Psi\colon\M_*\to\M_*$ by
  $\Psi(\omega)=\Phi^*(\omega)$. Then $\Phi=\Psi^*$, so $\Phi$ is
  normal.
\end{proof}

\begin{remark}\label{rk:graph-partition}
  Let~$R$ be a Feldman-Moore relation. It follows from the first part
  of the proof of \cite[Theorem~1]{fm1} that there is a countable
  family~$\{f_j\colon \delta_j\to\rho_j\colon j\ge0\}\subseteq \I(R)$
  such that $\{\Gr f_j\colon j\ge0\}$ is a partition of~$R$. Indeed,
  it is shown there that there are Borel sets~$\{D_j\colon j\ge1\}$
  which partition $R\setminus\Delta$ so that $D_j=\Gr f_j$, where
  $f_j\colon \pi_{1}(D_j)\to \pi_2(D_j)$ is a one-to-one map. Since
  $\Gr f_j$ and $\Gr(f_j^{-1})$ are both Borel sets, each~$f_j$ is
  in~$\I(R)$, and we can take $f_0$ to be the identity mapping on~$X$.
\end{remark}

\begin{theorem}\label{th_main}
  We have that $\{M(\phi) : \phi\in \fS\}$ coincides with the set of
  normal $\A$-bimodule maps on $\M$.
\end{theorem}
\begin{proof}
  Let~$\phi\in \fS$. %
  If~$a\in \Sigma$ and~$\beta,\gamma\in L^\infty(X,\mu)$, then by
  Lemma~\ref{lem:action},
  \begin{align*}
    M(\phi)\big(D(\beta)L(a)D(\gamma)\big)&=M(\phi)\big(L(c(\beta) \pw
    a\pw r(\gamma))\big)\\&=L(c(\beta)\pw \phi\pw a\pw
    r(\gamma)) = D(\beta)M(\phi)(L(a))D(\gamma)
    \end{align*}
    and~$M(\phi)$ is plainly linear, so $M(\phi)$ is an~$\A$-bimodule
    map.

  Let $r>0$ and let $T_n,T\in \M$ for $n\in\bN$ with
  $\|T_n\|,\|T\|\leq r$ and $T_n\to_{{\rm SOT}}T$.
  By Proposition~\ref{prop:continuity}(\ref{lem-obvious}), $s(T_n)\to_{\|\cdot\|_2} s(T)$,
  so by the $\|\cdot\|_2$ continuity of~$m(\phi)$,
  \[ m(\phi)(s(T_n))\to_{\|\cdot\|_2} m(\phi)(s(T));\]
  thus,
  \[ s(M(\phi)(T_n))\to_{\|\cdot\|_2} s(M(\phi)(T)).\]
  By Proposition~\ref{p_sconv},
  \[ M(\phi)(T_n)\to_{{\rm SOT}} M(\phi)(T).\] Since $L^2(R,\nu)$ is
  separable, Proposition~\ref{prop:continuity}(\ref{Mphi-cts}) and
  Lemma~\ref{lem-popsmith} show that $M(\phi)$ is normal.
  \medskip\goodbreak

  Now suppose that~$\Phi$ is a normal~$\A$-bimodule map on~$\M$. By
  Remark~\ref{rk:graph-partition}, we may write $R$ as a disjoint
  union $R = \bigcup_{k=1}^{\infty} F_k$, where~$F_k=\Gr f_k$
  and~$f_k\in \I(R)$, $k\in \bN$.  Let
  \[\phi : R\rightarrow \bC,\quad \phi(x,y) = \sum_{k\ge1}
  s(\Phi(V(f_k)))(x,y).\]
  Note that $\phi$ is well-defined since the sets
  $F_k$ are pairwise disjoint and, by Lemma~\ref{prop:bimod-symb}(1),
  $s(\Phi(V(f_k))) = s(\Phi(V(f_k)))\pw \chi_{F_k}$.
  It now easily follows that $\phi$ is measurable. Moreover,
  since each~$V(f_k)$ is a partial
  isometry (see Lemma~\ref{l_pin}(\ref{pisom})), 
  by \cite[Proposition~2.6]{fm2} we have
  \[ \|\phi\|_\infty=\sup_{k\ge1} \|s(\Phi(V(f_k)))\|_\infty
  \leq\sup_{k\ge1} \|\Phi(V(f_k))\|\leq \|\Phi\|;\]
  thus, $\phi$ is essentially bounded.

  We claim that $s(\Phi(T))=\phi\pw s(T)$ for every~$T\in\M$. First we
  consider the case $T=V(g)$ where $g\in \I(R)$. If we write $g_1=g$,
  then for~$m\ge2$ we can find $g_m\in \I(R)$ with graph $G_m=\Gr g_m$
  so that $R$ is the disjoint union $R=\bigcup_{m\ge1} G_m$. For
  example, we can define~$g_m$ to be the partial Borel isomorphism
  whose graph is~$F_{m-1}\setminus G_1$. Now let $\psi(x,y) =
  \sum_{m\ge1} s(\Phi(V(g_m)))(x,y)$, $(x,y)\in R$.  By
  Proposition~\ref{prop:bimod-symb}(2), we have $\phi\pw \chi_{F_k\cap
    G_m}=\psi\pw \chi_{F_k\cap G_m}$ for every~$k,m\ge1$, so
  $\phi=\psi$.  In particular, \[s(\Phi(V(g_1)))=\psi\pw
  \chi_{G_1}=\phi\pw\chi_{G_1}=\phi\pw s(V(g_1)).\] %
  Hence if~$T$ is in the left $\A$-module $\V$ generated
  by~$\{V(f)\colon f\in \I(R)\}$, then $s(\Phi(T))=\phi\pw s(T)$.
  On the other hand, by~\cite[Proposition~2.3]{fm2}, 
  $\V=\M_0(R,\sigma)$ and hence
  $\V$ is a strongly dense $*$-subalgebra of~$\M$.

  Now let $T\in \M$. By Kaplansky's Density Theorem, there exists a
  bounded net $\{T_\lambda\}\subseteq \V$ such that
  $T_\lambda\rightarrow T$ strongly. For every $\lambda$, we have that
  \[ s(\Phi(T_\lambda))=\phi\pw s(T_\lambda).\] By
  Proposition~\ref{prop:continuity}(\ref{lem-obvious}),
  $s(T_\lambda)\rightarrow_{\|\cdot\|_2} s(T)$ and, since $\phi\in
  L^{\infty}(R)$, we have \[\phi\pw
  s(T_\lambda)\rightarrow_{\|\cdot\|_2} \phi\pw s(T).\] On the other
  hand, since $\Phi$ is normal, $\Phi(T_\lambda)\to \Phi(T)$
  ultraweakly. Normal maps are bounded, so 
  $\{\Phi(T_\lambda)\}$ is a bounded net in~$\M$.  By
  Proposition~\ref{p_sconv}, $\Phi(T_\lambda)$ is strongly convergent.
  Thus, $\Phi(T_\lambda)\to \Phi(T)$ strongly. Since $\Phi(T)\in \M$,
  Proposition~\ref{p_sconv} yields
  \[s(\Phi(T_\lambda))\to_ {\|\cdot\|_2} s(\Phi(T)).\] By uniqueness
  of limits, $\phi\pw s(T)=s(\Phi(T))$. In particular, $\phi\pw
  s(T)\in \Sigma$ so $\phi$ is a Schur multiplier, and $\Phi(T)=
  L(\phi\pw s(T))=M(\phi)(T)$. It follows that $\Phi=
  M(\phi)$.
\end{proof}

\begin{remark}\label{r_autcb}
  The authors are grateful to Adam Fuller and David Pitts for bringing
  the following to our attention.  If~$(\M,\A)$ is a Cartan pair,
  then~$\A$ is norming for~$\M$ in the sense of~\cite{pss},
  by~\cite[Corollary 1.4.9]{cpz}. Hence by~\cite[Theorem 2.10]{pss},
  if $\phi$ is a Schur multiplier, then the map $M(\phi)$ is competely
  bounded with~$\|M(\phi)\|_{\cb}=\|M(\phi)\|$.
\end{remark}

We now show that up to isomorphism, the set of Schur multipliers of a
Cartan pair with respect to a Feldman-Moore coordinatisation~$R$
depends on~$(\M,\A)$, but not on~$R$.

\begin{proposition}\label{p_shpre}
  Let $(X_i,\mu_i,R_i,\sigma_i)$, $i = 1,2$, be isomorphic
  Feldman-Moore relations and let $\rho : X_1\to X_2$ be an
  isomorphism from $R_1$ onto $R_2$.  Then $\tilde{\rho}\colon
  a\mapsto a\circ \rho^{-2}$ is a bijection
  from~$\Sigma(R_1,\sigma_1)$ onto $\Sigma(R_2,\sigma_2)$, and an
  isometric isomorphism from $\fS(R_1,\sigma)$ onto
  $\fS(R_2,\sigma_2)$.
\end{proposition}
\begin{proof}
  It suffices to show that $\tilde
  \rho^{-1}(\Sigma(R_2,\sigma_2))\subseteq
  \Sigma(R_1,\sigma_1)$. Indeed, by symmetry we would then have
  $\tilde \rho(\Sigma(R_1,\sigma_1))\subseteq \Sigma(R_2,\sigma_2)$
  and could conclude that these sets are equal. Since $\tilde \rho$ is
  an isomorphism for the pointwise product, it then follows easily
  that $\tilde \rho(\fS(R_1,\sigma_1))=\fS(R_2,\sigma_2)$.

  For $i=1,2$, let~$s_i\colon \M(R_i,\sigma_i)\to
  \Sigma(R_i,\sigma_i)$ and $L_i=s_i^{-1}$ be the symbol map and the
  inverse symbol map for $R_i$, let $\nu_i$ be the right counting
  measure of~$R_i$ and let~$H_i=L^2(R_i,\nu_i)$.

  Let~$a\in \Sigma(R_2,\sigma_2)$ and let~$T=L_2(a)$. Since $T\in
  \M(R_2,\sigma_2)$, the Kaplansky density theorem gives a bounded net
  $\{T_\lambda\}\subseteq \M_0(R_2,\sigma_2)$ with $T_\lambda\to
  _{\mathrm{SOT}}T$. Let $a_\lambda=s_2(T_\lambda)$ and $a=s_2(T)$. By
  Proposition~\ref{prop:continuity}(6),
  \[a_\lambda\to a\quad\text{in~$H_2$}\] so if $U\colon H_2\to H_1$ is
  the unitary operator defined as in Remark~\ref{rk:unitary}, then
  \[(a_\lambda\circ \rho^2)\pw\eta = Ua_\lambda \to U a = (a\circ
  \rho^2)\pw\eta \quad\text{in~$H_1$}\] 
  where $\eta(x,y)=h(\rho(y))^{-1/2}$ and
  $h=\frac{d(\rho_*\mu_1)}{d\mu_2}$.  We can find a subnet, which can
  in fact be chosen to be a sequence $\{(a_n\circ \rho^2) \pw
  \eta\}$, that converges almost everywhere. Hence
  \[a_n\circ\rho^2\to a\circ\rho^2\quad\text{almost everywhere}.\]

  On the other hand, since $T_n$ converges to $T$ in the strong
  operator topology, $UT_nU^*$ converges to $UTU^*$
  strongly. Moreover, since $T_n\in \M_0(R_2,\sigma_2)$, Equation~(\ref{eq:unitary-action}) gives
  $s_1(UT_nU^*)=a_n\circ \rho^2$.
  Therefore
  \[a_n\circ\rho^2 =s_1(UT_n U^*)\to
  s_1(UTU^*)\quad\text{in~$H_1$}.\] %
  So $\tilde \rho^{-1}(a)=a\circ\rho^2 = s_1(UTU^*)\in
  \Sigma(R_1,\sigma_1)$.%
\end{proof}

\section{A class of Schur multipliers}\label{s_AR}

In this section, we examine a natural subclass of Schur multipliers on
$\M(R)$ which coincides, by a classical result of A. Grothendieck,
with the space of all Schur multipliers in the special case
$\M(R)=\B(\ell^2)$.  Throughout, we fix a Feldman-Moore relation
$(X,\mu,R,\sigma)$, and we write $\M(R) = \M(R,\sigma)$.  We first
recall some measure theoretic concepts \cite{a}.  A measurable subset
$E\subseteq X\times X$ is said to be \emph{marginally null} if there
exists a $\mu$-null set $M\subseteq X$ such that $E\subseteq (M\times
X)\cup (X\times M)$.  Measurable sets $E,F\subseteq X\times X$ are
called \emph{marginally equivalent} if their symmetric difference is
marginally null.  The set $E$ is called \emph{$\omega$-open} if $E$ is
marginally equivalent to a subset of the form $\cup_{k=1}^{\infty}
\alpha_k\times\beta_k$, where $\alpha_k,\beta_k\subseteq X$ are
measurable.

In the sequel, we will use some notions from operator space theory;
we refer the reader to~\cite{blm} and~\cite{pa} for background material.
Recall that every
element $u$ of the extended Haagerup tensor product 
$\A\otimes_{\eh}\A$ can be identified with a series \[u =
\sum_{i=1}^{\infty} A_i\otimes B_i,\] where $A_i,B_i\in \A$ and, for some constant $C > 0$,
we have
\[\left\|\sum_{i=1}^{\infty} A_i A_i^*\right\| \leq 
C \qand \left\|\sum_{i=1}^{\infty} B_i^* B_i\right\| \leq C\]
(the series being convergent in the weak* topology).  Let
$\A=\A(R)$. The element $u$ gives rise to a completely bounded
$\A'$-bimodule map $\Psi_u$ on $\B(L^2(R,\nu))$ defined by
\[\Psi_u(T) = \sum_{i=1}^{\infty} A_i T B_i, \quad T\in \B(L^2(R,\nu)).\]
For each~$T$, this series is $w^*$-convergent. Moreover, this element
$u\in \A\otimes_{\eh}\A$ also gives rise to a function $f_u\colon
X\times X\to \bC$, given by \[f_u(x,y) = \sum_{i=1}^{\infty}
a_i(x)b_i(y),\] where $a_i$ (resp. $b_i$) is the function in
$L^{\infty}(X,\mu)$ such that $D(a_i) = A_i$ (resp. $D(b_i) = B_i$),
$i\in \bN$. We write $u\sim \sum_{i=1}^{\infty} a_i\otimes b_i$.
Since
\begin{equation}\label{eq_C}
\left\|\sum_{i=1}^{\infty} |a_i|^2\right\|_{\infty} \leq 
C \qand \left\|\sum_{i=1}^{\infty} |b_i|^2\right\|_{\infty} \leq C,
\end{equation}
the function $f_u$ is well-defined up to a marginally null set.
Moreover, $f_u$ is \emph{$\omega$-continuous} in the sense that
$f_u^{-1}(U)$ is an $\omega$-open subset of $X\times X$ for every open
set $U\subseteq \bC$, and~$f_u$ determines uniquely the corresponding
element~$u\in \A\otimes_{\eh}\A$ (see~\cite{kp}).

\begin{definition}
  Given~$u\in \A\otimes_{\eh}\A$, we write
  \[ \phi_u\colon R\to \bC\] for the restriction of~$f_u$ to~$R$.
\end{definition}

In what follows, we identify $u\in \A\otimes_{\eh}\A$ with the
corresponding function~$f_u$, and write $\|\cdot\|_{\eh}$ for the norm
of $\A\otimes_{\eh}\A$.

\begin{lemma}\label{l_wd}
  If $E\subseteq X\times X$ is a marginally null set, then $E\cap R$
  is $\nu$-null.  Thus, given $u\in \A\otimes_{\eh}\A$, the function
  $\phi_u$  is well-defined as an element of
  $L^{\infty}(R,\nu)$.  Moreover, $\|\phi_u\|_{\infty}\leq
  \|u\|_{\eh}$.
\end{lemma}
\begin{proof}
  If $E \subseteq X\times M$, where $M\subseteq X$ is $\mu$-null, then
  $(E\cap R)_y = \emptyset$ if $y\not\in M$, and hence $\nu(E\cap R) =
  0$. Recall from Proposition~\ref{prop:nuinverse} that $\nu$ has the
  same null sets as the measure $\nu^{-1}$; so if $E\subseteq M\times
  X$, then $\nu(E\cap R)=0$. Hence any marginally null set is
  $\nu$-null.

  Since $\|u\|_{\eh}$ is the least possible constant $C$ so
  that~(\ref{eq_C}) holds, the set $\{(x,y)\in X\times X\colon
  |u(x,y)|>\|u\|_{\eh}\}$ is marginally null with respect to~$\mu$, so
  its intersection with~$R$ is $\nu$-null. Hence
  $\|\phi_u\|_{\infty}\leq \|u\|_{\eh}$.
\end{proof}

\begin{definition}
  Let
\[\fA(R) = \{\phi_u : u\in \A\otimes_{\eh}\A\}.\]
By virtue of Lemma~\ref{l_wd}, $\fA(R)\subseteq
L^{\infty}(R,\nu)$.
\end{definition}

\begin{lemma}\label{l_elt}
  If $a,b\in L^{\infty}(X,\mu)$ and $u = a\otimes b$, then for $T\in
  \M(R,\sigma)$ we have \[M(\phi_u)(T) = D(a)TD(b).\] In particular,
  $\phi_u\in \fS(R,\sigma)$.
\end{lemma}
\begin{proof}
  By Lemma~\ref{lem:action},
  \[s(D(a)TD(b))(x,y) = a(x)s(T)(x,y)b(y), \ \ \ (x,y)\in R.\] The
  claim is now immediate.
\end{proof}

\begin{lemma}\label{lem:pw-weak}
  Let $(Z,\theta)$ be a $\sigma$-finite measure space and let
  $\{f_k\}_{k\in \bN}$ be a sequence in $L^2(Z,\theta)$ such that

  (i) \ \ $f_k$ converges weakly to $f\in L^2(Z,\theta)$,

  (ii) \ $f_k$ converges (pointwise) almost everywhere to $g\in L^2(Z,\theta)$, and

  (iii) $\sup_{k\ge1}\|f_k\|_{\infty}<\infty$.

\noindent Then $f=g$.
\end{lemma}
\begin{proof} Let $\xi \in L^2(Z,\theta)$.  As $f_k$ converges weakly,
  $\{\|f_k\|_2\}$ is bounded.  Let $Y\subseteq Z$ be measurable with
  $\theta(Y)<\infty$. If we write $B=\sup_{k\ge1} \|f_k\|_\infty$, then
  \[|f_k\overline{\xi}\chi_Y|\leq B|\xi|\chi_Y.\] Since $B|\xi|\chi_Y$
  is integrable,
\begin{eqnarray*}
\langle f\chi_Y,\xi\rangle & = & \langle f,\chi_Y\xi\rangle =
\lim_{k\to \infty} \langle f_k,\chi_Y\xi\rangle=\lim_{k\to \infty}\int
f_k \overline{\xi}\chi_Y d\mu\\
& = & \int g\overline{\xi}\chi_Y\,d\mu=
\langle g\chi_Y,\xi\rangle
\end{eqnarray*}
by the Lebesgue Dominated Convergence
Theorem.  So $f\chi_Y = g\chi_Y$. Since $Z$ is $\sigma$-finite, this
yields $f = g$.
\end{proof}

\begin{theorem}\label{p_aeha}
  If  $u\in \A\otimes_{\eh}\A$, then
  $M(\phi_u)$ is the restriction of\/~$\Psi_u$ to $\M(R,\sigma)$
  and $\|M(\phi_u)\|\leq \|u\|_{\eh}$.
  Hence %
  \[ \fA(R)\subseteq \fS(R,\sigma).\]
\end{theorem}
\begin{proof}
  Let~$H=L^2(R,\nu)$, let $u\in \A\otimes_{\eh}\A$ and
  let~$\Psi=\Psi_u$; thus, $\Psi$ is a completely bounded map on 
  $\B(H)$. It is well-known that
  $\|\Psi\|_{\cb}=\|u\|_{\eh}$. We have $u\sim \sum_{i=1}^\infty
  a_i\otimes b_i$, for some $a_i,b_i\in \A$ with
  \[C=\max\left\{\Big\|\sum_{i=1}^\infty |a_i|^2 \Big\|_\infty,\
  \Big\|\sum_{i=1}^\infty |b_i|^2\Big\|_{\infty}\right\}<\infty.\] For $k\in
  \bN$, set $u_k = \sum_{i=1}^k a_i\otimes b_i$ and $\Psi_k =
  \Psi_{u_k}$.  By Lemma~\ref{l_elt}, $\Psi_k$ leaves $\M(R,\sigma)$
  invariant. Since $\Psi_k(T)\to _{w^*}\Psi(T)$ for each~$T\in
  \B(H)$, it follows that $\Psi$ also leaves~$\M(R,\sigma)$ invariant.

  Let $\Phi$ and $\Phi_k$ be the restrictions of $\Psi$ and $\Psi_k$,
  respectively, to $\M(R,\sigma)$.  Set $\phi_k = \phi_{u_k}$ for each
  $k\in \bN$.  Let $c\in \Sigma(R,\sigma)$ and let $T=L(c)$. By
  Lemma~\ref{l_elt}, $\phi_k\in \fS(R,\sigma)$,
  so $\phi_k\pw c\in \Sigma(R,\sigma)$ and
  \[ L(\phi_k \pw c) = \Phi_k(T) \to_{w^*}\Phi(T)
  \quad\text{as $k\to \infty$}.\] Hence for every $\eta\in
  H$, we have
\[\langle \phi_k \pw c,\eta\rangle = \langle L(\phi_k \pw c)(\chi_{\Delta}),\eta\rangle \rightarrow
\langle \Phi(T)(\chi_{\Delta}),\eta\rangle = \langle
s(\Phi(T)),\eta\rangle.\] So \[\phi_k \pw c \rightarrow
s(\Phi(T))\quad\text{ weakly in $L^2(R,\nu)$.}\] However, $u_k\to u$
marginally almost everywhere, so by Lemma~\ref{l_wd}, $\phi_k\to
\phi_u$ almost everywhere, and thus \[\phi_k \pw c\rightarrow \phi_u \pw
c\quad\text{almost everywhere.}\] Since \[\sup_{k\ge1}\|\phi_k \pw
c\|_\infty \leq C\|c\|_\infty<\infty,\] Lemma~\ref{lem:pw-weak} shows
that $\phi_u \pw c = s(\Phi(T))$.  Hence \[L(\phi_u\pw s(T)) = \Phi(T)
\in \M(R,\sigma)\] for every $T\in \M(R,\sigma)$, so $\phi_u$ is a Schur
multiplier and $M(\phi_u)=\Phi=\Psi|_{\M(R,\sigma)}$.  Since
$\|M(\phi_u)\|\leq\|M(\phi_u)\|_{\cb}$ (and in fact we have equality
by Remark~\ref{r_autcb}), this shows that $\|M(\phi_u)\| \leq
\|\Psi\|_{\cb} = \|u\|_{\eh}$.
\end{proof}

\section{Schur multipliers of the hyperfinite II$_1$-factor}\label{s_h21}

Recall the following properties of the classical Schur multipliers
of~$B(\ell^2)$.
\begin{enumerate}
\item Every symbol function is a Schur multiplier.
\item Every Schur multiplier is in $\fA(R)$.
\end{enumerate}
In this section, we consider a specific Feldman-Moore coordinatisation
of the hyperfinite II$_1$ factor, and show that in this context the
first property is satisfied but the second is not.

The coordinatisation we will work with is defined as follows. Let
$(X,\mu)$ be the probability space $X = [0,1)$ with Lebesgue measure
$\mu$, and equip~$X$ with the commutative group operation of addition
modulo $1$.  For~$n\in\bN$, let~$\bD_n$ be the finite subgroup of~$X$
given by
\[ \bD_n= \{\tfrac{i}{2^n} : 0\leq i \leq 2^n - 1\},\]
and let
\[\bD = \bigcup_{n=0}^{\infty} \bD_n.\]
The countable subgroup~$\bD$ acts on $X$ by translation;
let~$R\subseteq X\times X$ be the corresponding orbit equivalence
relation:
\[R = \{(x,x+r) : x\in X,\ r\in \bD\}.\] For~$r\in \bD$, define
\[\Delta_r=\{(x,x+r)\colon x\in X\}\]
and note that $\{\Delta_r\colon r\in\bD\}$ is a partition of~$R$.

Let ${\bf 1}$ be the $2$-cocycle on~$R$ taking the constant value $1$;
then~$(X,\mu,R,\bf 1)$ is a Feldman-Moore relation. Let~$\nu$ be the
corresponding right counting measure. Clearly, if~$E_r\subseteq
\Delta_r$ is measurable, then
$\nu(E)=\mu(\pi_1(E_r))=\mu(\pi_2(E_r))$. Hence if $E$ is a measurable
subset of~$R$, then for~$j=1,2$ we have
\begin{equation}\label{eq:nu}
  \nu(E) = \sum_{r\in \bD} \nu(E\cap \Delta_r) =
  \sum_{r\in \bD} \mu(\pi_j(E\cap \Delta_r)).
\end{equation}
It is well-known (see e.g.,~\cite{kr}) that $\R=\M(R,\bf1)$ is
(*-isomorphic to) the hyperfinite II$_1$-factor.

  For $1\leq i,j\leq 2^n$, define
\[\Delta^n_{ij} = \left\{\left(x, x + \frac{j-i}{2^n}\right) : \frac{i-1}{2^n} \leq x <
  \frac{i}{2^n}\right\}.\] Let~$\chi_{ij}^n$ be the characteristic
function of~$\Delta_{ij}^n$, and write
\[ \Sigma_n=\spn\{\chi_{ij}^n \colon 1\leq i,j\leq 2^n\}.\] Writing
$L$ for the inverse symbol map of~$R$, let~$\R_n\subseteq \R$ be given
by
\[ \R_n=\{L(a)\colon a\in \Sigma_n\}.\]
We also write
\[\iota_n\colon \R_n\to M_{2^n},\quad \sum_{i,j}
\alpha_{ij}L(\chi_{ij}^n)\mapsto (\alpha_{ij}).\] Recall that $\pw$
denotes pointwise multiplication of symbols. We write $A \odot B$ for
the Schur product of matrices $A,B\in M_k$ for some $k\in \bN$.
\begin{lemma}\leavevmode\label{lem:mus}
  \begin{enumerate}
  \item The set~$\{ L(\chi_{ij}^n)\colon 1\leq i,j\leq 2^n\}$ is a
    matrix unit system in~$\R$.
  \item The map~$\iota_n$ is a $*$-isomorphism. In particular,
    $\iota_n$ is an isometry.
  \item For~$a,b\in \Sigma_n$,
    we have
    \begin{enumerate}
    \item $a\pw b\in \Sigma_n$;
    \item $\iota_n(L(a\pw b))=\iota_n(L(a))\odot
      \iota_n(L(b))$; and
    \item $\|L(a\pw b)\|\leq \|L(a)\|\,\|L(b)\|$.
    \end{enumerate}
  \end{enumerate}
\end{lemma}
\begin{proof}
  Checking (1) is an easy calculation, and~(2) is then
  immediate. Statement~(3a) is obvious, and (3b) is plain from the
  definition of~$\iota_n$. It is a classical result of matrix theory
  that if~$A,B\in M_k$, then $\|A\odot B\|\leq
  \|A\|\,\|B\|$. Statement~(3c) then follows from~(2) and~(3b).
\end{proof}

Let $\tau\colon \R\to \bC$ be given by
\[\tau(L(a)) = \int_X a(x,x)\,d\mu(x).\]
Since $\nu=\nu^{-1}$, an easy calculation shows that~$\tau$ is a trace
on~$\R$.  

For~$a\in L^\infty(R,\nu)$, let
\[ \lambda_{ij}^n(a)=2^n \int_{(i-1)/2^n}^{i/2^n}
a(x,x+(j-i)/2^n)\,d\mu(x)\] be the average value of~$a$
on~$\Delta_{ij}^n$, and define
\[ E_n\colon \Sigma(R,\mathbf1)\to \Sigma_n,\quad a\mapsto \sum_{i,j} \lambda_{ij}^n(a)\chi_{ij}^n\]
and
\[ \bE_n\colon \R\to \R_n,\quad L(a)\mapsto L(E_n(a)).\]

\begin{lemma}\leavevmode\label{lem:cond}
  $\bE_n$ is the $\tau$-preserving conditional expectation
    of\/~$\R$ onto~$\R_n$. In particular, $\bE_n$ is norm-reducing.
\end{lemma}
\begin{proof}
  By~\cite[Lemma~3.6.2]{ss-book}, it suffices to show that~$\bE_n$ is
  a $\tau$-preserving $\R_n$-bimodule map. For~$a\in \Sigma(R,\mathbf
  1)$, we have
  \begin{align*} \tau(\bE_n(L(a))) &= \tau(L(E_n(a)))\\
    &= \int E_n(a)(x,x)\,d\mu(x)\\
    &= \sum_{i=1}^{2^n} \lambda_{ii}^n(a) \mu([ (i-1)/2^n,i/2^n))\\
    &= \tau(L(a)),
  \end{align*}
  so $\bE_n$ is $\tau$-preserving. For $b,c\in \Sigma_n$, a calculation gives
  \[ E_n(b *_{\mathbf1} a*_{\mathbf1} c) = b*_{\mathbf 1}E_n(a)*_{\mathbf 1}c,\]
  hence $\bE_n(BTC)=B\bE_n(T)C$ for $B,C\in \R_n$ and $T\in \R$.
\end{proof}

\begin{lemma}\label{lem:condconv}
  Let $a\in \Sigma(R,\mathbf1)$.
  \begin{enumerate}
  \item $\|E_n(a)\|_\infty \leq \|a\|_\infty$.
  \item $E_n(a)\to_{\|\cdot\|_2} a$ as $n\to \infty$.
  \end{enumerate}
\end{lemma}
\begin{proof}
  (1) follows directly from the definition of~$E_n$.%

  (2) For $T\in \R$, we have $\bE_n(T)\to_{\mathrm{SOT}} T$ as $n\to
  \infty$ (see e.g.,~\cite{ps}).  By
  Proposition~\ref{prop:continuity}(6),
  \[ E_n(a) = s(\bE_n(L(a))) \to_{\|\cdot\|_2} s(L(a))=a.\qedhere\]
\end{proof}

\begin{theorem}\label{th_inc}
  We have $\Sigma(R,\mathbf1)\subseteq \fS(R,\mathbf1)$. Moreover, if
  $a,b\in\Sigma(R,\mathbf1)$, then $\|L(a\pw b)\|\leq
  \|L(a)\|\|L(b)\|$.
\end{theorem}
\begin{proof}
  Let $a,b\in \Sigma(R,\mathbf1)$, and for $n\in\bN$, let $a_n=E_n(a)$
  and $b_n=E_n(b)$.  Lemmas~\ref{lem:mus} and~\ref{lem:cond} give
  \begin{equation}\label{eq2}
    \|L(a_n\pw b_n)\|
    \leq \|L(a_n)\|\,\|L(b_n)\|
    = \|\bE_n(L(a))\|\,\|\bE_n(L(b))\|
     \leq \|L(a)\|\|L(b)\|.
\end{equation}
On the other hand,
\begin{align*}
\|a_n\pw b_n - a\pw b\|_2
& \leq  \|a_n \pw (b_n - b)\|_2 + \|b\pw (a_n - a)\|_2\nonumber \\
& \leq  \|a_n\|_{\infty} \|(b_n - b)\|_2 + \|b\|_{\infty} \|(a_n - a)\|_2
\end{align*}
so by Lemma~\ref{lem:condconv},
\[ a_n\pw b_n\to_{\|\cdot\|_2} a\pw b.\] Let $T_n=L(a_n\pw
b_n)$. Since~$(T_n)$ is bounded by~(\ref{eq2}),
Proposition~\ref{p_sconv} shows that $(T_n)$ converges in the strong
operator topology, say to $T\in \R$, and \[a_n\pw
b_n=s(T_n)\to_{\|\cdot\|_2} s(T).\] Hence $a\pw b=s(T)\in
\Sigma(R,\mathbf1)$, so $a\in \fS(R,\mathbf1)$.

Since $T_n\to_{\mathrm{SOT}}T$, we have $\|T\|\leq \limsup_{n\to
  \infty} \|T_n\|$. Hence by~(\ref{eq2}),
\[\|L(a\pw b)\|\leq \limsup_{n\to \infty} \|L(a_n \pw b_n)\| \leq
\|L(a)\|\|L(b)\|.\qedhere\]
\end{proof}

\begin{remark}\label{remark:popsmith}
  For each masa $\A\subseteq \R$, Pop and Smith define a Schur product
  $\schur_\A \colon \R\times \R\to \R$ in~\cite{ps}. The proof of
  Theorem~\ref{th_inc} shows that for the specific Feldman-Moore
  coordinatisation $(X,\mu,R,\mathbf1)$ described above and the masa
  $\A=\A(R)\subseteq \R=\M(R,\mathbf1)$, if we identify operators
  in~$\R$ with their symbols, then Definition~\ref{d_sh} extends
  $\schur_\A$ to a map $\fS(R,\mathbf1)\times \R\to \R$. It is easy to
  see that this is a proper extension: the constant function
  $\phi(x,y)=1$ is plainly in~$\fS(R,\mathbf1)$, but $\phi$ is not the
  symbol of an operator in~$\R$ (\cite[Remark 3.3]{ps}). 
\end{remark}

\begin{corollary}\label{cor:inc}
  Let~$\R$ be the hyperfinite II$_1$ factor, and let~$\tilde \A$ be
  any masa in~$\R$. For any Feldman-Moore coordinatination $(\tilde
  X,\tilde \mu,\tilde R,\tilde \sigma)$ of the Cartan pair~$(\R,\tilde
  \A)$, we have $\Sigma(\tilde R,\tilde \sigma)\subseteq \fS(\tilde
  R,\tilde \sigma)$.
\end{corollary}
\begin{proof}
  By~\cite{cfw}, we have $(\R,\tilde \A)\cong (\R,\A)$. Hence by
  Theorem~\ref{thm:fmii1}, \[(\tilde X,\tilde \mu, \tilde R,\tilde
  \sigma)\cong (X,\mu,R,\mathbf1)\] via an isomorphism $\rho\colon
  \tilde X\to X$. Consider the map $\tilde \rho\colon a\mapsto a\circ
  \rho^{-2}$ as in Proposition~\ref{p_shpre}. By Theorem~\ref{th_inc},
  \[ \Sigma(\tilde R,\tilde \sigma)=\tilde \rho(\Sigma(R,\mathbf1))
  \subseteq \tilde\rho(\fS(R,\mathbf1))=\fS(\tilde
  R,\tilde\sigma).\qedhere\]
\end{proof}

In view of Theorem~\ref{th_inc} and Proposition~\ref{prop:sigma0}, it
is natural to ask the following question.

\begin{question}
  Does the inclusion $\Sigma(R,\sigma)\subseteq \fS(R,\sigma)$ hold
  for an arbitrary Feldman-Moore relation $(X,\mu,R,\sigma)$?
\end{question}

We now turn to the inclusion
\[ \fA(R)\subseteq \fS(R,\sigma)\] established in
Section~\ref{s_AR}. While these sets are equal in the classical
case, we will show that in the current context this inclusion is
proper.

\newcommand{\Rd}{R}

For~$D\subseteq\bD$, we define
\[\Delta(D)=\bigcup_{r\in D} \Delta_r.\]
Note that $\Delta(D)$ is marginally null only if $D=\emptyset$, and
its characteristic function $\chi_{\Delta(D)}$ is a ``Toeplitz''
idempotent element of $L^\infty(R,\nu)$.

\begin{proposition}\label{prop:dyad-A(Rd)}\leavevmode
\begin{enumerate}
\item If\/ $\emptyset\ne D\subsetneq \bD$ and either $D$ or\/
  $\bD\setminus D$ is dense in $[0,1)$, then the characteristic
  function~$\chi_{\Delta(D)}$ is not in\/~$\fA(\Rd)$.

  \item  Let $0\ne \phi\in L^\infty(\Rd)$ and
    \[ E=\{ r\in \bD\colon \phi|_{\Delta_r}=0\
    \mu\text{-a.e.}\}.\]If $E$ is dense in $[0,1)$, then $\phi\not\in
    \fA(\Rd)$.
  \end{enumerate}
\end{proposition}
\begin{proof}
  (1) Suppose first that $\bD\setminus D$ is dense in $[0,1)$ and, by
  way of contradiction, that $\chi_{\Delta(D)}\in\fA(\Rd)$. There is
  an element $\sum_{i=1}^\infty a_i\otimes b_i\in \A\otimes_{\eh}\A$ and a
  $\nu$-null set $N\subseteq \Rd$ such
  that \[\chi_{\Delta(D)}(x,y)=\sum_{i=1}^\infty a_i(x)b_i(y)\ \text{for all
    $(x,y)\in \Rd\setminus N$.}\]
  Let $f\colon X\times X\to \bC$ be the extension of $\chi_{\Delta(D)}$
  which is defined (up to a marginally null set) by
  \[f(x,y)=\sum_{i=1}^\infty a_i(x)b_i(y)\ \text{for marginally almost
    every~$(x,y)\in X\times X$}.\] By~\cite[Theorem~6.5]{eks}, $f$ is
  $\omega$-continuous. Hence the set \[F=f^{-1}(\bC\setminus \{0\})\]
  is $\omega$-open. Since $D\ne\emptyset$ and $\Delta(D)\subseteq F$,
  the set~$F$ is not marginally null.  So there exist Borel sets
  $\alpha,\beta\subseteq [0,1)$ with non-zero Lebesgue measure so that
  $\alpha\times \beta\subseteq F$. For~$j=1,2$, let $N_j=\pi_j(N)$. By
  equation~(\ref{eq:nu}), $\mu(N_j)=0$. Let $\alpha'=\alpha\setminus
  N_1$ and $\beta'=\beta\setminus N_2$; then $\alpha'$ and $\beta'$
  have non-zero Lebesgue measure, and hence the set \[\beta'-\alpha' =
  \{ y-x\colon x\in \alpha',\,y\in \beta'\}\] contains an open
  interval by Steinhaus' theorem, so it intersects the dense
  set~$\bD\setminus D$. So there exist~$r\in \bD\setminus D$ and $x\in
  \alpha'$ with $x+r\in \beta'$. Now
  \[ (x,x+r)\in F \setminus \Delta(D),\]
  so
  \[0\ne f(x,x+r)=\chi_{\Delta(D)}(x,x+r)=0,\] a contradiction. So
  $\chi_{\Delta(D)}\not\in \fA(\Rd)$ if $D\ne\emptyset$ and $\bD\setminus D$ is dense in
  $[0,1)$.

  If $D\ne \bD$ and~$D$ is dense in $[0,1)$ then
  $\chi_{\Delta(\bD\setminus D)}\not\in \fA(\Rd)$; since $\fA(\Rd)$ is
  a linear space containing the constant function $1$, this shows
  that $1-\chi_{\Delta(\bD\setminus
    D)}=\chi_{\Delta(D)}\not\in\fA(\Rd)$.\medskip

  (2) The argument is similar. If $\phi\in\fA(\Rd)$ then there is a
  $\nu$-null set $N\subseteq \Rd$ such that $\phi(x,y)=\sum_{i=1}^{\infty}
  a_i(x)b_i(y)$ for all $(x,y)\in \Rd\setminus N$ where $\sum_{i=1}^{\infty}
  a_i\otimes b_i\in \A\otimes_{\eh}\A$, and $\phi(x,y)=0$ for all
  $(x,y)\in \Rd\setminus N$ with the property $y-x\in E$. Let $f\colon
  [0,1)^2\to \bC$, $f(x,y)=\sum_{i=1}^{\infty} a_i(x)b_i(y)$, $x,y\in [0,1)$.
  Then $f$ is non-zero
  and $\omega$-continuous, so $f^{-1}(\bC\setminus\{0\})$ contains
  $\alpha'\times \beta'$ where $\alpha',\beta'$ are sets of non-zero
  measure so that $(\alpha'\times \beta')\cap N = \emptyset$. Hence
  $\beta'-\alpha'$ contains an open interval of $[0,1)$, and
  intersects the dense set $E$ in at least one point $r\in \bD$; so
  there is $x\in [0,1)$ such that $(x,x+r)\in (\alpha'\times\beta')\cap
  (\Rd\setminus N)$. Then $0=\phi(x,x+r)=f(x,x+r)\ne0$, a
  contradiction.
\end{proof}

\begin{corollary}\label{cor:Rdinclusion}
  The inclusion $\fA(R)\subseteq \fS(R,\mathbf1)$ is proper.
\end{corollary}
\begin{proof}
  Since~$\Delta=\Delta(\{0\})$, Proposition~\ref{prop:dyad-A(Rd)}
  shows that $\chi_\Delta\not\in \fA(R)$.  It is easy to check (as in
  Lemma~\ref{lem:cond}) that the Schur multiplication map
  $M(\chi_{\Delta})$ is the conditional expectation of $\R$ onto $\A$,
  so $\chi_\Delta\in \fS(R,\mathbf1)$.
\end{proof}

\begin{corollary}\label{c_noncbs}
  Let~$(\tilde X,\tilde \mu,\tilde R,\tilde \sigma)$ be a
  Feldman-Moore relation and suppose that $\M(\tilde R,\tilde \sigma)$
  contains a direct summand isomorphic to the hyperfinite~{\rm II}$_1$
  factor. Then the inclusion~$\fA(\tilde R)\subseteq \fS(\tilde
  R, \tilde \sigma)$ is proper.
\end{corollary}
\begin{proof}
  Let~$P$ be a central projection in~$\M(\tilde R,\tilde \sigma)$ so
  that~$P\M(\tilde R, \tilde\sigma)$ is (isomorphic to) the
  hyperfinite II$_1$ factor~$\R$. It is not difficult to verify
  that~$\A_P=P\A(\tilde R)$ is a Cartan masa in~$\R$ (see the
  arguments in the proof of~\cite[Theorem~1]{fm2}). By~\cite{cfw}, the
  Cartan pair~$(\R,\A_P)$ is isomorphic to the Cartan pair~$(\R,\A)$
  considered throughout this section. It follows from
  Theorem~\ref{thm:fmii1} that there is a Borel isomorphism $\rho :
  \tilde X\to X_0 \cup [0,1)$ (a disjoint union) with $\rho^2(\tilde
  R) = R_0\cup \Rd$ (again, a disjoint union), where $R_0\subseteq
  X_0\times X_0$ is a standard equivalence relation and $R$ is the
  equivalence relation defined at the start of the present section.
  It is easy to check that $\rho^2(\fA(\tilde R)) = \fA(R_0\cup \Rd)$.
  We may thus assume that $\tilde X = X_0 \cup [0,1)$ and $\tilde R =
  R_0\cup \Rd$.  %

  Now suppose that $\fS(\tilde R,\tilde \sigma) = \fA(\tilde R)$.  Let $P
  = P([0,1))$. Given $\phi\in \fS(\Rd)$, let $\psi : \tilde R\to
  \bC$ be its extension defined by letting $\psi(x,y) = 0$ if
  $(x,y)\in R_0$. Then
  \[M(\psi)(T\oplus S) = P M(\psi)(T\oplus S) P = M(\phi)(T)\oplus 0,
  \quad T\in \M(\Rd).\] So $\psi\in \fS(R,\mathbf1)$ and hence
  $\psi\in \fA(\tilde R)$. It now easily follows that $\phi\in
  \fA(\Rd)$, contradicting Corollary~\ref{cor:Rdinclusion}.
\end{proof}

In fact, the only Toeplitz idempotent elements of~$\fS(\Rd):=\fS(\Rd, \mathbf1)$ are
trivial.  To see this, we first explain how $\fS(\Rd)$ can be obtained from
multipliers of the Fourier algebra of a measured groupoid.  We refer
the reader to~\cite{rbook,r} for basic notions and results about
groupoids.

The set $\G=X\times\mathbb D$ becomes a groupoid under the partial product
\[ (x,r_1)\cdot (x+r_1,r_2)=(x,r_1+r_2)\quad\text{for $x\in X$,
  $r_1,r_2\in \bD$}\]
where the set of composable pairs is
\[\G^2=\{\big((x_1,r_1),(x_2,r_2)\big): x_2=
x_1+r_1\}\] and inversion is given by
\[ (x,t)^{-1}=(x+t, -t).\] The domain and range maps in this case are
$d(x,t)=(x,t)^{-1}\cdot(x,t)=(x+t,0)$ and
$r(x,t)=(x,t)\cdot(x,t)^{-1}=(x,0)$, so the unit space, $\G_0$, of
this groupoid, which is the common image of $d$ and $r$, can be
identified with $X$.  Let $\lambda$ be the Haar, that is, the
counting, measure on $\mathbb D$.  The groupoid $\G$ can be equipped
with the Haar system $\{\lambda^x:x\in X\}$, where
$\lambda^x=\delta_x\times \lambda$ and $\delta_x$ is the point mass
at~$x$.
\newcommand{\nug}{\nu_{\G}}

Recall that $\mu$ is Lebesgue measure on $X$. Consider the measure
$\nug$ on~$\G$ given by $\nug=\mu\times\lambda=\int\lambda^xd\mu(x)$.
Since~$\mu$ is translation invariant and~$\lambda$ is invariant under
the transformation $t\mapsto -t$, it is easy to see that
$\nug^{-1}=\nug$, where $\nug^{-1}(E)=\nug(\{e^{-1}\colon e\in
E\})$.%
Therefore $\G$ with the above Haar system and the measure~$\mu$
becomes a measured groupoid.

Consider the map
\[\theta:\Rd\to X\times \bD,\quad \theta(x,x+r)=(x,r),\quad x\in X,\ r\in \bD.\]
Clearly $\theta$ is a continuous bijection (here~$\bD$ is equipped
with the discrete topology). We claim the measure~$\theta_*\nu \colon
E\mapsto \nu(\theta^{-1}(E))$ is equal to~$\nug$, where, as
before,~$\nu$ is the right counting measure for the Feldman-Moore
relation~$(X,\mu,R,\mathbf1)$. Indeed, for~$E\subseteq\G$, we have
\begin{align*}
  (\theta_*\nu)(E)
  &=
  \nu(\theta^{-1}(E))\\
  &=
  \sum_{r\in \bD} \mu(\pi_1(\theta^{-1}(E)\cap \Delta_r))\quad\text{by equation~(\ref{eq:nu})}
  \\&=
  \sum_{r\in \bD} \mu(\pi_1(E\cap (X\times \{r\}))) = (\mu\times \lambda)(E)=\nug(E)
\end{align*}
since it is easily seen that $\pi_1(\theta^{-1}(E)\cap
\Delta_r)=\{x\in X\colon (x,r)\in E\}$.
It follows that  the operator \[U:L^2(\Rd,\nu)\to L^2(\G,\nug),\quad \xi\mapsto \xi\circ \theta^{-1}\] is unitary.

Let $C_c(\G)$ be the space of compactly supported continuous functions
on $\G$. This becomes a $*$-algebra with respect to the convolution
given by
\[(f\ast g)(x,t)=\sum_{r\in \bD} f(x,r)g(x+r,t-r),\]
and involution given by $f^*(x,t)=\overline{f(x+t,-t)}$.

Let $\Reg$ be the representation of $C_c(\G)$ on the Hilbert space
$L^2(\G,\nug)$ given for~$\xi,\eta\in L^2(\G,\nug)$ by
\begin{align*}
\langle \Reg(f)\xi,\eta\rangle& =  \int f(x,t)\xi((x,t)^{-1}(y,s))\overline{\eta(y,s)}d\lambda^{r(x,t)}(y,s)d\lambda^u(x,t)d\mu(u)\\
&=\int f(x,t)\xi(x+t,s-t)\overline{\eta(x,s)}d\lambda(s)d\lambda(t)d\mu(x)\\
&=\int f(x,t)\xi(x+t,s-t)\overline{\eta(x,s)}d\lambda(t)d\nug(x,s)
\end{align*}
hence 
\[(\Reg(f)\xi)(x,s)=
  \int f(x,t)\xi(x+t,s-t)d\lambda(t)=\sum_t f(x,t)\xi(x+t,s-t).\]
In~\cite[Section 2.1]{r}, the von Neumann algebra $\VN(\G)$ of $\G$
is defined to be the bicommutant $\Reg(C_c(\G))''$.

If $f\in C_c(\G)$, then $f\circ \theta$ has a band limited support and
for~$\xi\in L^2(R,\nu)$, we have
\begin{align*}
  (U^*\Reg(f)U\xi)(x,x+t)&=\sum_sf(x,s)\xi(x+s,x+t)\\
  &=\sum_sf(\theta(x,x+s))\xi(x+s,x+t)\label{equivalence}\\
  &=(L(f\circ \theta)\xi)(x,x+t).
\end{align*}
Hence
\begin{equation}
  U^*\Reg(f)U = L(f\circ \theta)\label{equivalence}
\end{equation}
and so $\VN(\G)$ is spatially isomorphic to $\M(\Rd)$.

The von Neumann algebra $\VN(\G)$ is the dual of the Fourier algebra
$A(\G)$ of the measured groupoid $\G$, which is a Banach algebra of
complex-valued functions on $\G$.  If the operator $M_\phi$ on $A(\G)$
of multiplication by the function $\phi\in L^\infty(\G)$ is bounded,
then its adjoint $M_\phi^*$ is a bounded linear map on $\VN(\G)$.
Moreover, in this case we have $M_\phi^*\Reg(f)=\Reg(\phi f)$, for
$f\in C_c(\G)$.  The function $\phi$ is then called a multiplier of
$A(\G)$ \cite{r} and we write $\phi\in MA(\G)$.  If the map $M_\phi$
is also completely bounded then $\phi$ is called a completely bounded
multiplier of $A(\G)$ and we write $\phi\in M_0A(\G)$.  By
equation~(\ref{equivalence}) and Remark~\ref{r_autcb}, we have
\begin{equation}
  \phi\in M_0A(\G)
  \iff \phi\circ \theta\in \fS(R,\mathbf1).\label{eq:renault}
\end{equation}

We are now ready to prove the following statement:

\begin{proposition}\label{prop:dyadmult}
  If $D\subseteq \bD$, then the following are equivalent:
  \begin{enumerate}
  \item The function $\chi_{\Delta(D)}\in L^\infty(\Rd,\nu)$ is in
    $\fS(\Rd)$.
  \item The function $\chi_D\in \ell^\infty(\bD)$ is in the
    Fourier-Stieltjes algebra $B(\bD)$ of $\bD$.
  \item $D$ is in the coset ring of $\bD$.
  \end{enumerate}
\end{proposition}
\begin{proof}

  To see that $(1)$ and~$(2)$ are equivalent, observe that if
  $\pi:\G\to \mathbb D$, $(x,t)\mapsto t$ is the projection
  homomorphism of $\G$ onto $\mathbb D$, then
  \[\chi_{\Delta(D)}=\chi_{D}\circ\pi\circ\theta.\]
  Moreover, since~$\bD$ is commutative, we have $B(\bD)=M_0 A(\bD)$. So
  \begin{align*}
    \chi_D\in B(\bD)&\iff \chi_D\in M_0A(\bD)
    \\&\iff \chi_D\circ \pi \in M_0A(\G)\ \text{by~\cite[Proposition~3.8]{r}}\\
    &\iff \chi_{\Delta(D)}=\chi_D\circ \pi\circ \theta\in \fS(R,\mathbf1)\ \text{by~(\ref{eq:renault}).}
  \end{align*}
  The equivalence of $(2)$ and $(3)$ follows
  from~\cite[Chapter~3]{rudin-fag}.
\end{proof}

\begin{theorem}\label{th_ch}
  The only elements of $\fA(\Rd)$ of the form $\chi_{\Delta(D)}$ for
  some $D\subseteq \bD$ are $0$ and $1$.
\end{theorem}
\begin{proof}
  If $\chi_{\Delta(D)}\in \fA(\Rd)$ then $\chi_{\Delta(D)}\in \fS(\Rd)$ by
  Proposition~\ref{p_aeha}, so $D$ is in the coset ring of $\bD$ by
  Proposition~\ref{prop:dyadmult}. All proper subgroups of $\bD$ are
  finite, so $D$ is in the ring of finite or cofinite subsets of
  $\bD$. Hence either $\bD\setminus D$ or $D$ is dense in $[0,1)$,
  so either $D=\emptyset$ or $D = \bD$ by
  Proposition~\ref{prop:dyad-A(Rd)}.
\end{proof}

\begin{remark}
  We note that there exist non-trivial idempotent elements of
  $\fA(R)$. For example, if~$\alpha,\beta$ are measurable subsets
  of~$X$, then the characteristic function of $(\alpha\times
  \beta)\cap \Rd$ is always idempotent. Note that the sets of the form
  $(\alpha\times\beta)\cap \Rd$ are not unions of full diagonals
  unless they are equivalent to either $\Rd$ or the empty set.
\end{remark}

\end{document}